\documentclass[letter]{amsart}

\usepackage[margin = 1.4in]{geometry}

\newcommand{\bitem}{\begin{compactitem}}
	\newcommand{\eitem}{\end{compactitem}}
\newcommand{\ds}{\displaystyle}
\newcommand{\ov}[1]{\overline{#1}}
\newcommand{\orient}[1]{\overleftrightarrow #1}
\newcommand{\oE}{\orient{E}}
\newcommand{\oG}{\orient{G}}
\newcommand{\vG}{\vec{G}}
\newcommand{\mU}{\mathcal{U}}
\newcommand{\mV}{\mathcal{V}}

\usepackage{graphicx,paralist,comment} 
\usepackage{amsmath}
\usepackage{amssymb}
\usepackage{amsthm}
\usepackage{amsfonts}
\usepackage{mathtools}
\usepackage{hyperref}
\usepackage{textgreek}
\usepackage{setspace}
\usepackage{ragged2e}
\usepackage{bbold}
\usepackage{tikz}
\usepackage{tikz-cd}
\usetikzlibrary{decorations.markings, arrows.meta}
\usepackage[usestackEOL]{stackengine}
\usepackage{array}
\usepackage[tone]{tipa}
\usepackage{fancyhdr}
\usepackage{ltablex}
\usepackage{xcolor}
\definecolor{mblue}{RGB}{ 35, 149, 174 }

\makeatletter
\newcommand{\oset}[3][0ex]{%
	\mathrel{\mathop{#3}\limits^{
			\vbox to#1{\kern-2\ex@
				\hbox{$\scriptstyle#2$}\vss}}}}
\makeatother

\def\multiset#1#2{\ensuremath{\left(\kern-.3em\left(\genfrac{}{}{0pt}{}{#1}{#2}\right)\kern-.3em\right)}}
\tikzstyle{v}=[circle, draw, fill=black,
inner sep=0pt, minimum width=2pt]
\tikzstyle{empty} = [fill = none, draw=none, rectangle]
\makeatother
\newtheorem{theorem}{Theorem}[section]

\newtheorem{lemma}[theorem]{Lemma}
\newtheorem{corollary}[theorem]{Corollary}
\newtheorem{cor}[theorem]{Corollary}
\theoremstyle{definition}
\newtheorem{definition}[theorem]{Definition}
\newtheorem{remark}[theorem]{Remark}
\newtheorem{example}[theorem]{Example}
\newtheorem{fact}[theorem]{Fact}

\newcommand{\al}{\alpha}
\newcommand{\be}{\beta}
\newcommand{\ga}{\gamma}

\newcommand{\mA}{\mathcal{A}}
\newcommand{\mB}{\mathcal{B}}
\newcommand{\mG}{\mathcal{G}}

\newcommand{\bN}{\mathbb{N}}

\newcommand{\cyc}{\text{cyc}}
\newcommand{\coc}{\text{coc}}
\newcommand{\Cyc}{\text{Cyc}}
\newcommand{\Acy}{\text{Acy}}
\newcommand{\cc}{\text{cc}}

\newcommand{\Four}{\text{Four}}

\newcommand{\la}[1]{{\oset{\leftarrow}{#1}}}
\newcommand{\ra}[1]{{\oset{\rightarrow}{#1}}}
\newcommand{\ba}[1]{{\oset{\leftrightarrow}{#1}}}
\newcommand{\ua}[1]{{\oset{\circ}{#1}}}

\allowdisplaybreaks
\usepackage{blindtext}

\newcommand{\fig}[3]{\begin{figure}[h!]\begin{center}\includegraphics[#1]{#2.pdf}\end{center}\caption{#3}\label{fig:#2}\end{figure}}

\graphicspath{{./Figures/}{../Figures/}}

\begin{document}
	
\title{Subgraphs versus Orientations:\\ Infinite families of equidistributions}
\author{Oliver Bernardi and Jonathan J. Fang}
\date{\today}

\begin{abstract}
A classical enumerative result states that, given a graph $G$ and a vertex $u$, the number of connected subgraphs of $G$ is equal to the number of orientations of $G$ such that every vertex can reach $u$ by a directed path. 
We show that this result is an instance of a much broader set of enumerative identities between subgraphs and orientations corresponding to various connectivity constraints. Namely, given two sets of pairs of vertices $A=\{(u_i,v_i),~i\in[k]\}$ and $B=\{(u_i',v_i'),~i\in[\ell]\}$, we consider the orientations $\al$ of $G$ such that adding the elements of $A$ and $B$ as additional directed edges to $\al$ gives an orientation $\al'$ in which $v_i$ cannot reach $u_i$ for all $i\in[k]$, but $v_i'$ can reach $u_i'$ for all $i\in[\ell]$. We show that this set of orientations is equinumerous to a set of subgraphs satisfying the ``same" connectivity constraints defined in terms of $A$ and $B$.

We also extend our results to the enumeration of equivalence classes of orientations satisfying such connectivity constraints. Precisely, we consider the equivalence classes under cycle reversal, cocycle reversal, or cycle-cocycle reversal. We show that the equivalences classes are equinumerous to some sets of subgraphs defined by connectivity and acyclicity constraints.
\end{abstract}

\maketitle
	
\section{Introduction}
This article is concerned with a set of enumerative identities between subgraphs and orientations.
Consider a finite undirected graph $G=(V,E)$ (loops and multiple edges are allowed). By \emph{subgraph} of $G$ we mean a graph of the form $H=(V,F)$ for some subset of edges $F\subseteq E$. 
A well-known result states that, given a vertex $u$, the number of orientations of $G$ such that every vertex can reach $u$ by a directed path is equal to the number of connected subgraphs of $G$~\cite{Gioan:degree-sequences}. Another classical result is that, given two vertices $u$ and $v$, the number of orientations such that $v$ can reach $u$ is equal to the number of subgraphs such that $u$ and $v$ are connected~\cite{Tutte:GT}. And, of course, the total number of orientations, $2^{|E|}$, is equal to the number of subgraphs. What if these identities were all part of a big family? This is the main result we establish in the present article. 

Let us set the vocabulary needed to state our results. Consider a graph $G=(V,E)$ and two arbitrary sets of pairs of vertices $A,B\subseteq V^2$.
Given an orientation $\al$ of $G$, we consider the oriented graph $\vec G\equiv \vec G(A,B;\al)$ obtained from $\al$ by adding the pairs in $A$ and $B$ as additional directed edges. We say that the orientation $\al$ is \emph{$(A,B)$-valid} if
\bitem
\item[(a)] for all $(u,v)$ in $A$, the vertex $v$ cannot reach $u$ by a directed path of $\vec G$, and
\item[(b)] for all $(u,v)$ in $B$, the vertex $v$ can reach $u$ by a directed path of $\vec G$.
\eitem
Given a subgraph $H=(V,F)$ we consider the oriented graph $\vec G\equiv \vec G(A,B;F)$ obtained from $H$ by replacing every edge in $H$ by two directed edges (in opposite direction) and adding the pairs in $A$ and $B$ as additional directed edges. We say that the subgraph $H$ is \emph{$(A,B)$-valid} if $\vec G$ satisfies the conditions (a) and (b) above.

Our main result is that for every graph $G$ and sets $A$ and $B$, the $(A,B)$-valid orientations are equinumerous to the $(A,B)$-valid subgraphs. Let us examine some special cases for illustration. For $A=\{(u,v)\}$ and $B=\emptyset$, our result implies that the number of orientations of $G$ such that there is no directed path from $v$ to $u$ is equal to the number of subgraphs for which $u$ and $v$ are not connected by a path. For $A=\emptyset$ and $B=\{(u,v)\mid v\in V\}$, our result implies the classical result mentioned above: the orientations of $G$ such that every vertex can reach $u$ by a directed path are equinumerous to the connected subgraphs of $G$. Figure~\ref{fig:example-subgraphVorient} gives another illustration.

\fig{width=.9\linewidth}{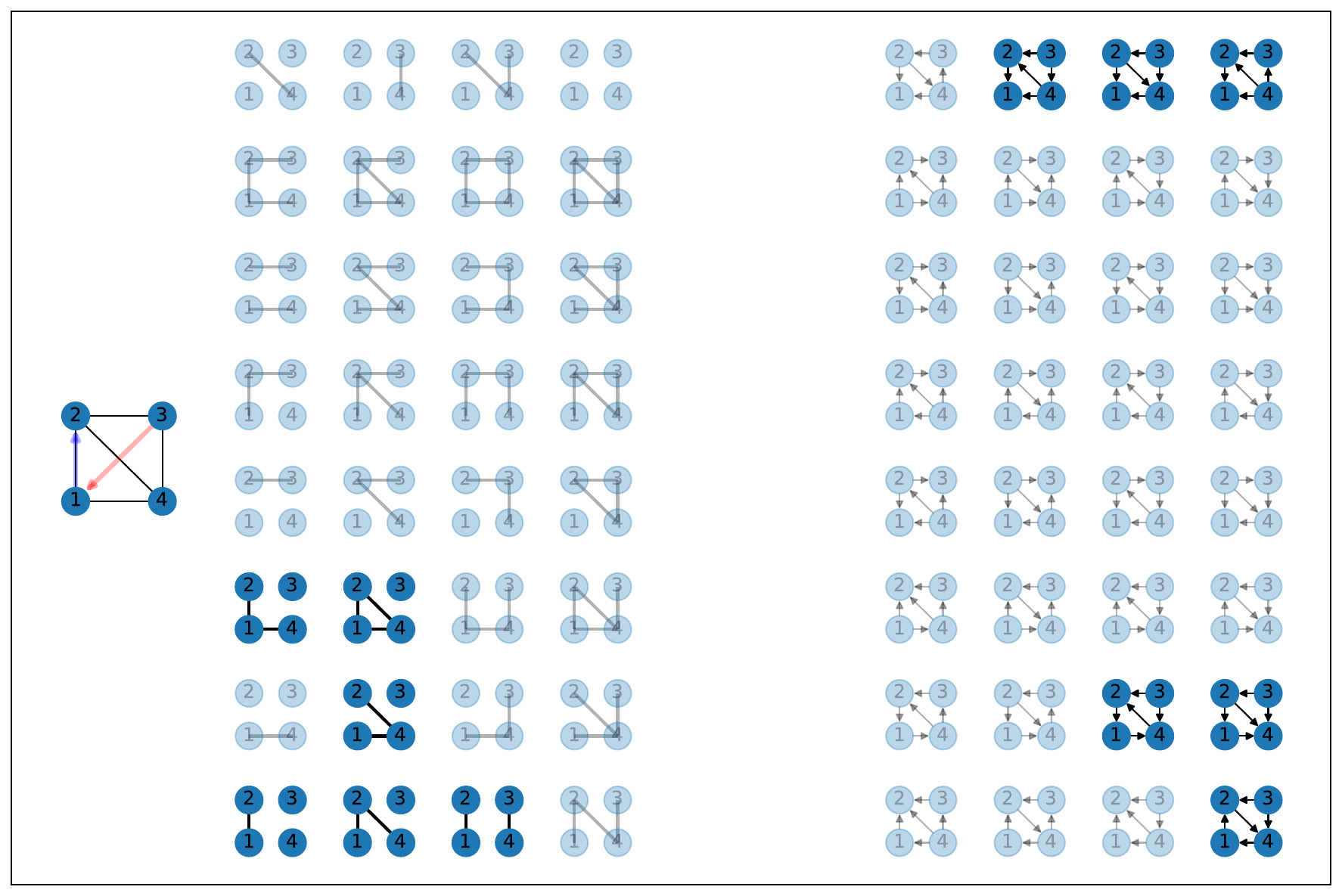}{The $(A,B)$-valid subgraphs and orientations (6 among 32). The graph $G$ has 5 edges on the vertex set $\{1,2,3,4\}$, and $A=\{(3,1)\}$ (red) and $B=\{(1,2)\}$ (blue).}

A good way to phrase the conditions (a) and (b) in a unified manner, and to interpolate between the case of orientations and the case of subgraphs is through the formalism of \emph{fourientations}~\cite{Backman:fourientation-evaluations,Backman:fourientation-activities} (a.k.a. \emph{biorientations}~\cite{OB-EF:dangulations}). A \emph{fourientation} of a graph $G=(V,E)$ is an assignment $\phi$ of one of four configurations for each edge $e\in E$ indicating how $e$ can be used on a path from a vertex to another: 
\bitem
\item either $e$ can be traversed in both directions; we say that $e$ is \emph{2-way} in this case, 
\item or $e$ can only be traversed in one specified direction; we say that $e$ is \emph{1-way} in this case, 
\item or $e$ cannot be traversed at all; we say that $e$ is \emph{0-way} in this case. 
\eitem
We say that an edge in a fourientation is \emph{solid} if it is either 0-way or 2-way. Clearly, orientations correspond to fourientations without solid edges, while subgraphs correspond to fourientations such that every edge is solid. We say that a fourientation $\phi$ of $G=(V,E)$ is \emph{$(A,B)$-valid} if the fourientation $\vec G$ obtained from $\phi$ by adding the 1-way edges in $A$ and $B$ satisfies the conditions (a), (b) above.
Our result above can then be generalized as follows: given any set of edges $S\subseteq E$, the number of $(A,B)$-valid fourientations having $S$ as their set of solid edges is independent of the set $S$. Indeed, comparing these cardinalities for $S=\emptyset$ and $S=E$ yields the previously stated identities between $(A,B)$-valid orientations and $(A,B)$-valid subgraphs.

Another set of results we establish is about equivalence classes of $(A,B)$-valid orientations (and fourientations). 
Following~\cite{Gioan:degree-sequences}, we say that two orientations of $G$ are \emph{cycle-reversal equivalent} if one can be obtained from the other by repeatedly flipping some directed cycles. A classical result is that the cycle-reversal equivalence classes of orientations of $G$ are equinumerous to the number of forests of $G$ (subgraphs without cycles).
We show that for any sets $A$, $B$ the cycle-reversal equivalence classes of $(A,B)$-valid orientations are equinumerous to the $(A,B)$-valid forests of $G$. 
Similar results are established for the cocycle-reversal equivalence, and the cycle-cocycle-reversal equivalence. In particular we give new proofs for the classical case $A=\emptyset$ and $B=\emptyset$ that cycle-reversal (resp. cocycle-reversal, cycle-cocycle) equivalence classes of orientations of $G$ are equinumerous to the number of forests (resp. connected subgraphs, spanning trees) of $G$.


\section{Notation}
In this section we set our notation and review some definitions about fourientations. 

We say \emph{graph} to mean finite undirected graph, where loops and multiple edges are allowed. We say \emph{digraph} to mean directed graph, and \emph{arc} to mean directed edge. We say that a vertex~$v$ in a digraph can \emph{reach} a vertex $u$ if there is a directed path from $v$ to $u$. We say that an arc $a=(u,v)$ in a digraph $G=(V,E)$ is \emph{cyclic} if it belongs to a directed cycle (i.e. $v$ can reach $u$) and \emph{acyclic} otherwise (i.e. $v$ cannot reach $u$). A non-empty set of arcs $S\subseteq A$ of the digraph $G$ is a \emph{directed cocycle} if there exists a partition $V=V_1\cup V_2$ such that $S$ is the set of arcs of $G$ having their initial vertex in $V_1$ and terminal vertex in $V_2$, and $G$ has no arc with initial vertex in $V_2$ and terminal vertex in $V_1$. It is well known 
that an arc is acyclic if and only if it belongs to a directed cocycle.

For a graph $G=(V,E)$, we let $\oG=(V,\orient{E})$ be the digraph obtained from $G$ by replacing each edge by two arcs in opposite direction. As explained above, a \emph{fourientation} of $G$ is an assignment of one among four possible configurations for each edge, either 0-way, 1-way (two possibilities) or 2-way. If we fix a conventional orientation $\ra{e}$ for an edge $e\in E$ of $G$, then we denote by $\ra{e}$ and $\la{e}$ the two 1-way configurations of $e$, by $\ua{e}$ the 0-way configuration, and by $\ba{e}$ the 2-way configuration.
Note that the fourientations of $G$ are naturally in correspondence with the digraphs of the form $D=(V,F)$ with $F\subseteq \orient{E}$. 
(By convention, even loops of $G$ have four possible configurations in fourientations, so there are $4^{|E|}$ fourientations of $G$.) 
 For a fourientation $\phi$, we denote by $\ra{\phi}$ the corresponding digraph $D=(V,F)$. By \emph{directed paths} (resp. \emph{directed cycles}, \emph{directed cocycles}) of a fourientation $\phi$, we mean directed paths (resp. directed cycles, directed cocycles) in the digraph $\ra{\phi}$.


Given a fourientation $\phi$ of a graph $G=(V,E)$ and two sets of pairs of vertices $A\subseteq V^2$ and $B\subseteq V^2$ (which can be thought as arcs on $V$), we denote by $\vG(A,B;\phi)$ the digraph obtained from $\ra{\phi}$ by adding the arcs in $A$ and $B$.
\begin{definition}
Given a graph $G=(V,E)$ and two sets of pairs of vertices $A, B\subseteq V^2$, we say that a fourientation $\phi$ of $G$ is $(A,B)$-valid if the digraph $\vec G=\vG(A,B;\phi)$ satisfies the conditions (a) and (b) stated in the introduction.
\end{definition}

\begin{remark}\label{rk:valid}
A fourientation $\phi$ of $G$ is $(A,B)$-valid if and only if every arc in $A$ is acyclic and every arc in $B$ is cyclic in $\vG(A,B;\phi)$.
\end{remark}

Recall that the \emph{solid edges} of a fourientation $\phi$ of $G=(V,E)$ are the edges which are either 0-way or 2-way. For a set $S\subseteq E$, we denote by $\Four_{G,S}(A,B)$ the set of $(A,B)$-valid fourientations with set of solid edges equal to $S$. Recall that the orientations of $G$ correspond to fourientations without solid edges, while the subgraphs of $G$ correspond to fourientations such that every edge is solid. 


\section{Identities for subgraphs, orientations and fourientations}
In this section we state, illustrate and prove our main result, which is a set of enumerative identities between classes of subgraphs, orientations, and fourientations.

\subsection{Main result}
Our main result is stated in terms of the cardinality of classes of fourientations with a given solid edge set $S$ (they interpolate between subgraphs and orientations). 
\begin{theorem}\label{thm:main}
	For any graph $G=(V,E)$ and sets $A,B\subseteq V^2$, the number $|\Four_{G,S}(A,B)|$ of $(A,B)$-valid fourientations with solid edge set $S$ is independent of $S$.
\end{theorem}
By comparing the case $S = \varnothing$ to the case $S= E$ in Theorem~\ref{thm:main}, we get:
\begin{corollary}\label{cor:main}
	Let $G = (V,E)$ be a graph. For any sets $A,B\subseteq V^2$ the number of $(A,B)$-valid orientations is equal to the number of $(A,B)$-valid subgraphs.
\end{corollary}
Let us illustrate Theorem~\ref{thm:main} and Corollary~\ref{cor:main} by considering various connectivity constraints $A,B$.

\begin{example}\label{ex:rootcon}
Let $G = (V,E)$ be a graph, and let $u\in U\subseteq V$. 
By Corollary~\ref{cor:main} with $A = \lbrace (u,v)\mid v\in V\setminus U\rbrace$ and $B = \lbrace (u,v)\mid v\in U\rbrace$, the number of orientations such that for every vertex in $U$ can reach $u$ and every vertex in $V\setminus U$ cannot reach $u$ is equal to the number of subgraphs such that the connected component containing $u$ is $U$.
For $U=V$, we get that the number of orientations such that every vertex can reach $u$ is equal to the number of connected subgraphs. Applying Theorem~\ref{thm:main} to this case shows that the number of fourientations such that every vertex can reach $u$ is $2^{|E|}$ times the number of connected subgraphs.
\end{example}

\begin{example}\label{ex:gencon}
	Let $G = (V,E)$, and let $X,Y\subseteq V$ be such that $X\cup Y = V$. 
	Let $A = \varnothing$ and $B = \lbrace (x,y)\mid x\in X,~ y\in Y\rbrace$. Applying Corollary~\ref{cor:main} gives that the number of orientations of $G$ such that each vertex $y\in Y$ can reach at least one vertex in $X$ and every vertex $x\in X$ can be reached by at least one vertex in $Y$ is equal to the number of subgraphs of $G$ in which each connected component contains at least one vertex in $X$ and one vertex in $Y$. 
\end{example}

%

\begin{remark}
Given Corollary \ref{cor:main} one could wonder if a simpler result holds. Namely, given a graph $G=(V,E)$ and sets $A,B\subseteq V^2$, one could wonder if the following numbers are equal
\bitem
\item the number $o$ of orientations of $G$ such that for all $(u,v)\in A$, $v$ cannot reach $u$, while for all $(u,v)\in B$, $v$ can reach $u$,
\item the number $s$ of subgraphs of $G$ such that for all $(u,v)\in A$, $v$ cannot reach $u$, while for all $(u,v)\in B$, $v$ can reach $u$.
\eitem
This is not the case in general. For instance, for the triangle graph $G=K_3$ with vertex set $\{1,2,3\}$ and the sets $A=\emptyset$, $B=\{\{1,2\},\{2,3\}\}$ one finds $o=3$ and $s=4$. 
\end{remark}

\subsection{Enumerative consequences}
Let us reframe Theorem~\ref{thm:main} in terms of cyclic and acyclic edges. Let $G=(V,E)$ be a graph, and let $\oG=(V,\oE)$ be the corresponding digraph.
For an orientation $\al$ of $G$, we define $\Acy(\al)\subseteq \oE$ and $\Cyc(\al)\subseteq \oE$ as the sets of acyclic arcs and cyclic arcs, respectively. We extend this definitions to fourientations: for an fourientation $\phi$ of a graph $G$, we define $\Acy(\phi)\subseteq \oE$ and $\Cyc(\phi)\subseteq \oE$ as the sets of acyclic 1-way edges and cyclic 1-way edges, respectively.
The following is a consequence of Theorem~\ref{thm:main}.

\begin{cor}
Let $G=(V,E)$ be a graph, and let $\oG=(V,\oE)$ be the corresponding digraph. Let $(y_a)_{a\in \oE}$ and $(z_a)_{a\in \oE}$ be variables. Then 
\begin{equation}\label{eq:mainthm}
\sum_{\al \textrm{ orientation of }G~}\prod_{a\in \Acy(\al)}(1+y_a)\prod_{b\in \Cyc(\al)}(1+z_b)=\sum_{\phi \textrm{ fourientation of }G~}\prod_{a\in \Acy(\phi)}y_a\prod_{b\in \Cyc(\phi)}z_b.
\end{equation}
\end{cor} 
	
\begin{proof} For every subsets $A,B\subseteq \oE$ we can compare the coefficient of $\prod_{a\in A}y_a\prod_{b\in B}z_b$ in the left-hand side and right-hand side of~\eqref{eq:mainthm}. Denoting by $G\setminus\! A\setminus\! B$ the subgraph of $G$ obtained by deleting the edges in $E$ corresponding to the arcs in $A\cup B$, one gets
$$\bigg[\prod_{a\in A}y_a\prod_{b\in B}z_b\bigg]LHS=\# (A,B)\textrm{-valid orientations of }G\setminus\! A\setminus\! B,$$
and
$$\bigg[\prod_{a\in A}y_a\prod_{b\in B}z_b\bigg]RHS=\# (A,B)\textrm{-valid subgraphs of }G\setminus\! A\setminus\! B.$$
Since these quantities are equal for all sets $A,B$ by Theorem~\ref{eq:mainthm} (see Remark~\ref{rk:valid}), we conclude that~\eqref{eq:mainthm} holds.
\end{proof}	

A fourientation is called \emph{acyclic} (resp. \emph{totally cyclic}) if all its 1-way edges are acyclic (resp. cyclic). Note that by setting $y_a=1$ and $z_a=0$ for all $a\in \oE$ in~\eqref{eq:mainthm} one gets
\begin{equation} 
\# \textrm{ acyclic fourientations of }G=\sum_{\al\textrm{ orientation of }G}2^{|\Acy(\al)|},
\end{equation}
while by setting $y_a=0$ and $z_a=1$ for all $a\in \oE$ one gets
\begin{equation}\label{eq:tot-cyc}
\# \textrm{ totally cyclic fourientations of }G=\sum_{\al\textrm{ orientation of }G}2^{|\Cyc(\al)|}.
\end{equation}
From~\eqref{eq:tot-cyc}, one can derive the following result, first stated in~\cite{Archer:acyclic-strong-descents}, which follows implicitly from a recurrence obtained by Wright~\cite{Wright:strong-digraphs} and independently Robinson~\cite{Robinson:counting-digraphs}.
\begin{cor}[\cite{Archer:acyclic-strong-descents}] \label{cor:Ira}
For all $n\in \bN$, let $t_n$ be the number of totally cyclic orientations of the complete graph $K_n$ (strongly connected tournaments on $n$ vertices), and let $s_n$ be the number of strongly connected fourientations of $K_n$ (strongly connected simple digraphs on $n$ vertices). These numbers are related by the following equation:
\begin{equation}\label{eq:Ira}
\frac{1}{\ds 1-\sum_{n=1}^\infty2^{{n\choose 2}}t_n\frac{x^n}{n!}}=\exp\left(\sum_{n=1}^\infty s_n\frac{x^n}{n!}\right).
\end{equation}
\end{cor}
\begin{proof}
The main idea is that the the left-hand side of~\eqref{eq:Ira} can be interpreted as the generating function of oriented complete graphs with a weight of 2 per cyclic arc, while the right-hand side can be interpreted as the generating function of totally cyclic fouriented complete graphs. These are shown to be equal by applying~\eqref{eq:tot-cyc} to complete graphs.

Let us give some more details. First, consider the exponential generating function 
$$A(x)=\sum_{\al \in \mA} 2^{|\Cyc(\al)|}\frac{x^{|\al|}}{|\al|!},$$ 
where the sum is over the set $\mA$ of all the orientations of complete graphs, and for such an orientation $\al\in \mA$ we denote by $|\al|$ the number of vertices. For an orientation $\al\in \mA$ one can consider the set of strongly connected components $C_1,\ldots,C_r$. Each component $C_i$ can be identified with a totally cyclic orientation of a complete graph. Contracting these components gives an acyclic orientation of $K_r$, hence a total order on the components $C_1,\ldots,C_r$. Thus, the class $\mA$ can be identified with the class of finite sequences of totally cyclic orientations of complete graphs. This implies (by standard arguments~\cite{Flajolet:analytic}) the following generating function equation:
$$A(x)=\frac{1}{\ds 1-\sum_{n=1}^\infty2^{{n\choose 2}}t_n\frac{x^n}{n!}}.$$

Consider now the exponential generating function 
$$B(x)=\sum_{\be \in \mB} \frac{x^{|\be|}}{|\be|!},$$ 
where the sum is over the set $\mB$ of all the totally cyclic fourientations of complete graphs, and $|\be|$ denotes the number of vertices.
Any fourientation $\be\in\mB$ decomposes as a set of (strongly) connected components, which implies (by a standard argument, namely the exponential formula~\cite{Flajolet:analytic}) the following equation:
 $$B(x)=\exp\left(\sum_{n=1}^\infty s_n\frac{x^n}{n!}\right).$$

Finally, applying~\eqref{eq:tot-cyc} to the complete graph $K_n$ gives $[x^n]A(x)=[x^n]B(x)$. We thus obtain $A(x)=B(x)$ which completes the proof of~\eqref{eq:Ira}.
\end{proof}

\subsection{Proof of Theorem~\ref{thm:main}}
In this subsection we prove Theorem~\ref{thm:main}. We fix a graph $G=(V,E)$ and the sets $A,B\subseteq V^2$ throughout. 
It suffices to show that for every set $S\subset V$ and edge $e\in E\setminus S$, 
\begin{equation}\label{eq:suffice}
|\Four_{G,S}(A,B)| = |\Four_{G,S\cup e}(A,B)|.
\end{equation} 
We now fix $S$ and $e\in E\setminus S$, and proceed to prove~\eqref{eq:suffice}.

Let $\Omega = \lbrace \ua{e}, \ra{e}, \la{e}, \ba{e}\rbrace$ be the set of possible configurations of $e$, and let $(\Omega, \preceq)$ be the poset given by $\ua{e}\prec \ra{e}\prec\ba{e}$ and $\ua{e}\prec \la{e}\prec\ba{e}$. 
	For a fourientation $\phi$ of $G\setminus e$ 
	and a configuration $c\in \Omega$ we denote by $\phi\cup c$ the fourientation of $G$ obtained by giving configuration $c$ to $e$. We then define the function $f_\phi:\Omega\rightarrow \lbrace 0,1\rbrace,$ by: 
	$$f_\phi(c) = \begin{cases}
		1&\text{if }\phi\cup c\text{ is }(A,B)\text{-valid},\\
		0&\text{otherwise}.\\
	\end{cases}$$

Note that 	
$$|\Four_{G,S}(A,B)| =\sum_\phi f_\phi(\ra{e}) + f_\phi(\la{e})~\textrm{ and }~|\Four_{G,S\cup e}(A,B)|=\sum_\phi f_\phi(\ua{e}) + f_\phi(\ba{e}),$$ 
where the sums are over the fourientations $\phi$ of $G\setminus e$ with solid edge set $S$. 
Thus, in order to prove~\eqref{eq:suffice} it suffices to prove that for every fourientation $\phi$ of $G\setminus e$,
\begin{equation}\label{eq:suffice2}
f_\phi(\ra{e}) + f_\phi(\la{e}) = f_\phi(\ua{e}) + f_\phi(\ba{e}).
\end{equation}

We now fix an orientation $\phi$ of $G\setminus e$, and proceed to prove~\eqref{eq:suffice2}. For a configuration $c\in \Omega$, we denote by $D(c)$ the directed graph $\vG(A,B;\phi\cup c)$. 
Recall that a fourientation is $(A,B)$-valid if it satisfies Conditions (a) and (b). We say that $c\in\Omega$ is
\bitem
\item \textit{acyclic-valid} if $\phi\cup c$ satisfies (a), that is, every arc in $A$ is acyclic in $D(c)$,
\item \textit{cyclic-valid} if $\phi\cup c$ satisfies (b), that is, every arc in $B$ is cyclic in $D(c)$.
\eitem

We first establish two easy lemmas.

	\begin{lemma}\label{lem:order}
		Let $c\preceq d$ in the poset $(\Omega, \preceq)$. If $c$ is cyclic-valid, then $d$ is cyclic-valid. If $d$ is acyclic-valid, then $c$ is acyclic-valid.
	\end{lemma}

\begin{proof}
The digraph 	$D(d)$ is obtained from $D(c)$ by adding some arcs. Hence any directed path in $D(c)$ is a directed path in $D(d)$. 
Hence if $c$ is is cyclic-valid, then $d$ is also cyclic-valid; and if $d$ is acyclic-valid, then $c$ is acyclic-valid.
\end{proof}

\begin{lemma}\label{lem:upgrade}
		Let $c\in\lbrace \la{e}, \ra{e}\rbrace$ such that $\phi \cup c$ is $(A,B)$-valid. If $c$ is cyclic in $D(c)$, then $\phi \cup{\ba{e}}$ is $(A,B)$-valid. If $c$ is acyclic in $D(c)$, then $\phi \cup{\ua{e}}$ is $(A,B)$-valid. 
\end{lemma}
\begin{proof}
Suppose first that the arc $c$ is cyclic in $D(c)$. The digraph $D(\ba{e})$ is obtained from $D(c)$ by adding the arc opposite to $c$. Since $c$ is cyclic, this addition does not change the strongly connected components. Since $c$ is acyclic-valid, the arcs in $A$ have endpoints in different components of $D(c)$, hence $\ba{e}$ is acyclic-valid. Moreover $\ba{e}$ is cyclic-valid by Lemma~\ref{lem:order}, hence $\phi \cup{\ba{e}}$ is $(A,B)$-valid.
		
Suppose now that the arc $c$ is acyclic in $D(c)$. Then no directed cycle containing an arc $b\in B$ in $D(c)$ can also contain $c$. Thus $\ua{e}$ is cyclic-valid. Moreover $\ua{e}$ is acyclic-valid by Lemma~\ref{lem:order}, hence $\phi \cup \ua{e}$ is $(A,B)$-valid. 
\end{proof}

We will now prove~\eqref{eq:suffice2}, starting with the following:\smallskip

\noindent \textbf{Claim 1.} $f_\phi(\ba{e}) + f_\phi(\ua{e}) = 2$ if and only if $f_\phi(\ra{e}) + f_\phi(\la{e}) = 2$.
\smallskip

Suppose that $f_\phi(\ba{e}) + f_\phi(\ua{e}) = 2$. Since $\ba{e}$ is acyclic-valid, Lemma~\ref{lem:order} implies that $\ra{e}$ and $\la{e}$ are both acyclic-valid. 
Since $\ua{e}$ is cyclic-valid, Lemma~\ref{lem:order} implies that $\ra{e}$ and $\la{e}$ are both cyclic-valid. Thus, $f_\phi(\ra{e}) + f_\phi(\la{e}) = 2$.

Now suppose that $f_\phi(\ra{e}) + f_\phi(\la{e}) = 2$. 
By Lemma~\ref{lem:order}, $\ba{e}$ is cyclic-valid and $\ua{e}$ is acyclic-valid. We need to show that $\ba{e}$ is acyclic-valid and $\ua{e}$ is cyclic-valid.
If $\ba{e}$ is not acyclic-valid, then there exists an arc $a = (u,v)\in A$ such that there is a directed path from $v$ to $u$ in $D(\ba{e})$. But then there is a directed path from $v$ to $u$ in either $D(\ra{e})$ or $D(\la{e})$, which is impossible since $\ra{e}$ and $\la{e}$ are acyclic-valid. It remains to show that $\ua{e}$ is cyclic-valid, that is, for every arc $b=(u,v)\in B$ we must show that there is a directed path from $v$ to $u$ in $D(\ua{e})$. If there is a directed path from $v$ to $u$ not containing $e$, then we are done. Otherwise there exists a directed path in $D(\ra{e})$ containing $\ra{e}$ and a directed path in $D(\la{e})$ containing $\la{e}$. But then, as shown in Figure~\ref{fig:proof-path}(a), there is a directed path from $v$ to $u$ avoiding $e$, which contradicts our assumption. This completes the proof of Claim 1.
		
		
\fig{width = \textwidth}{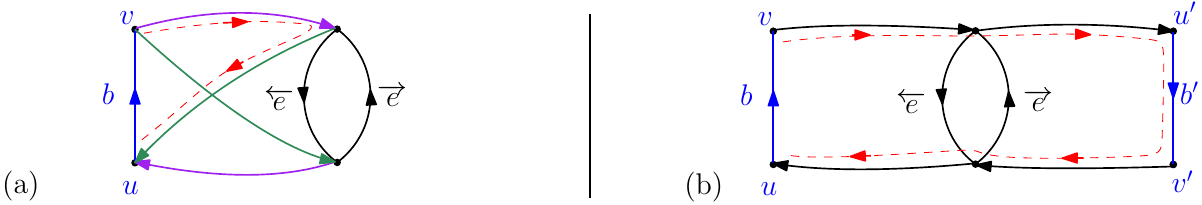}{(a) The case where there are directed paths from $v$ to $u$ containing $\ra{e}$, and one containing $\la{e}$. The dashed path avoids $e$. (b) The case where there is a path from $v$ to $u$ in $D(\la{e})$, and a path from $v'$ to $u'$ in $D(\ra{e})$. The dashed path avoids $e$.}


\noindent \textbf{Claim 2.} $f_\phi(\ba{e}) + f_\phi(\ua{e}) >0$ if and only if $f_\phi(\ra{e}) + f_\phi(\la{e}) >0$.
\smallskip
		
Suppose that $f_\phi(\la{e}) + f_\phi(\ra{e})>0$. Since any 1-way arc must be either cyclic or acyclic, Lemma~\ref{lem:upgrade} implies $f_\phi(\ua{e}) + f_\phi(\ba{e}) > 0$. 

Suppose conversely that $f_\phi(\ua{e}) + f_\phi(\ba{e}) > 0$. Consider first the situation where $\phi\cup\ba{e}$ is $(A,B)$-valid. By Lemma~\ref{lem:order}, both $\la{e}$ and $\ra{e}$ are acyclic-valid. We need to prove that one of them is cyclic-valid. If this is not the case, then there are arcs $b=(u,v)$ and $b'=(u',v')$ in $B$ such that there is no path from $v$ to $u$ in $D(\ra{e})$, and no path from $v'$ to $u'$ in $D(\la{e})$. Since there are such paths in $D(\ba{e})$, we conclude that there is a path from $v$ to $u$ using $\la{e}$ in $D(\la{e})$, and a path from $v'$ to $u'$ using $\ra{e}$ in $D(\ra{e})$. But then, as shown in Figure~\ref{fig:proof-path}(b), there is a directed path from $v$ to $u$ in $D(\ra{e})$ avoiding $e$ entirely, which is a contradiction. 

Consider now the situation where $\phi\cup\ua{e}$ is $(A,B)$-valid. By Lemma~\ref{lem:order}, both $\la{e}$ and $\ra{e}$ are cyclic-valid. We need to prove that one of them is acyclic-valid. If this is not the case, then there exist arcs $a,a'\in A$ such that there is a directed cycle containing $a$ and $\ra{e}$ in $D(\ra{e})$ and a directed cycle containing $a'$ and $\la{e}$ in $D(\la{e})$. But then there is a directed cycle containing $a$ and $a'$ avoiding $e$ entirely, which contradicts the acyclic-validity of $\ua{e}$. This completes the proof of Claim 2.
		
The above claims imply~\eqref{eq:suffice2}, which completes the proof of Theorem~\ref{thm:main}.


	
\section{Equivalence classes of $(A,B)$-valid orientations under cycle-cocycle reversal}
In this section we study equivalence classes of $(A,B)$-valid orientations up to cycle or cocycle reversals. We fix the graph $G = (V,E)$ and the sets $A,B\subseteq V^2$ throughout. 

\subsection{Results for equivalence classes of orientations and fourientations}	
By \emph{reversing a directed cycle} (resp. \emph{cocycle}) of a fourientation $\phi$ of $G$ we mean reversing all the 1-way edges on that directed cycle (resp. cocycle), and leaving the 2-way edges (resp. 0-way edges) unchanged. We call \emph{$(A,B)$-cocycle} of $\phi$ a cocycle of $\vG(A,B;\phi)$ which does not contain any arc from $A\cup B$. 
	
\begin{definition} 
	We say that two fourientations $\phi$ and $\phi'$ of $G$ are \textit{cycle equivalent} (resp. \emph{cocycle equivalent}) if $\phi$ can be obtained from $\phi'$ by repeatedly reversing some directed cycles (resp. $(A,B)$-cocycles). 
	We say that two fourientations $\phi$ and $\phi'$ are \textit{cycle-cocycle equivalent} if $\phi$ can be obtained from $\phi'$ by repeatedly reversing some directed cycles and/or $(A,B)$-cocycles.
\end{definition}

\begin{lemma}\label{lem:ABequiv}
If $\phi$ and $\phi'$ are cycle-cocycle equivalent fourientations of $G$, then $\phi$ is $(A,B)$-valid if and only if $\phi'$ is $(A,B)$-valid. 
\end{lemma}

\begin{proof} Reversing a directed cycle or $(A,B)$-cocycle of a fourientation $\phi$ of $G$ corresponds to reversing a cycle or a cocycle of the contracted digraph $\vG(A,B;\phi)/T$, where $T$ is the set of 2-way edges of $\phi$. Such a cycle or cocycle reversal does not change the status of arcs between cyclic and acyclic, hence in particular it does not change whether the arcs in $A$ are acyclic and those in $B$ are cyclic (equivalently, it does not change the $(A,B)$-validity of the fourientation).
\end{proof}

By Lemma~\ref{lem:ABequiv} it makes sense to consider cycle (resp. cocycle, cycle-cocycle) equivalence classes of $(A,B)$-valid fourientations. 

\begin{example}
Consider the graph $G$ in Figure~\ref{fig:equivalence-classes}, and the sets $A=\emptyset$, $B=\{(u,v)\}$. There are 11 $(A,B)$-valid orientations. These orientations are partitioned in 6 cycle equivalence classes (resp. 10 cocycle equivalence classes, 5 cycle-cocycle equivalence classes).
\end{example}

\fig{width=\linewidth}{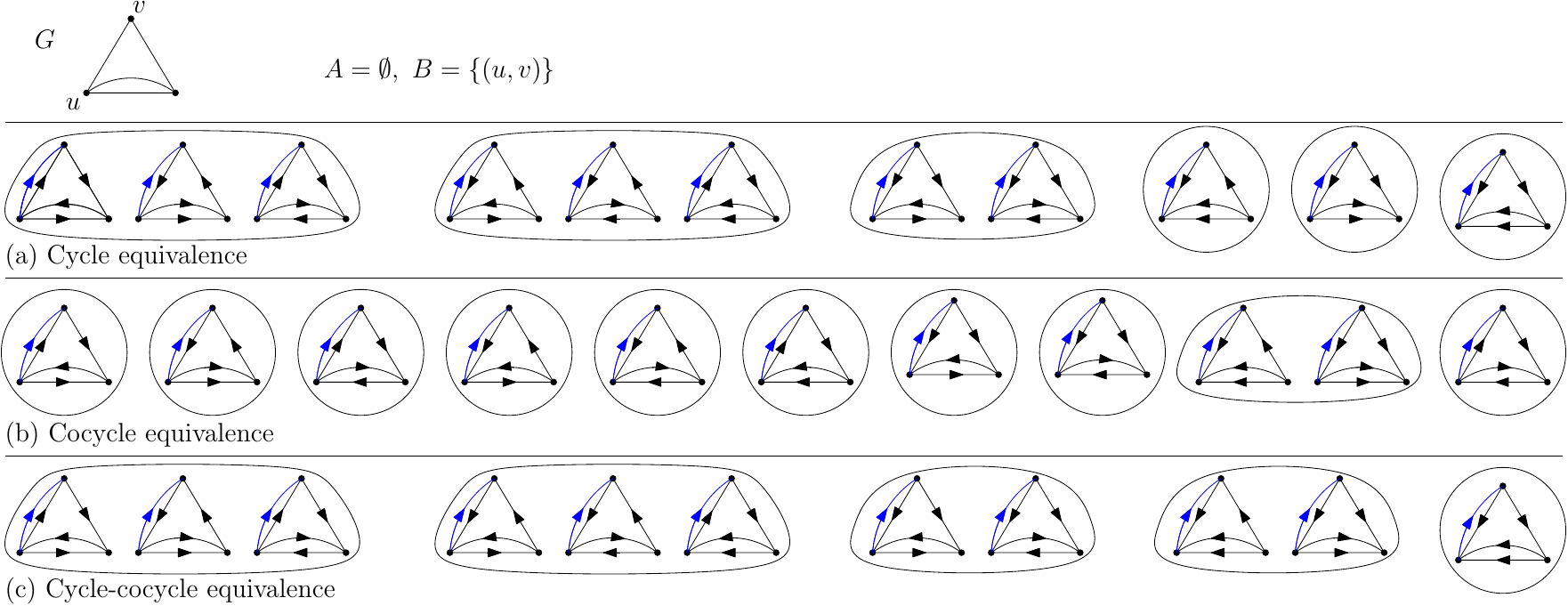}{(a) Cycle equivalence classes of $(A,B)$-valid orientations. (b) Cocycle equivalence classes of $(A,B)$-valid orientations. (c) Cycle-cocycle equivalence classes of $(A,B)$-valid orientations. The cyclic constraint $(u,v)\in B$ is represented by a blue arc.}

	%
	%
	%
	
\begin{definition}\label{def:valid-classes}
	\begin{compactenum}[(i)]
		\item An \textit{$(A,B)$-valid cycle class} is a cycle equivalence class of $(A,B)$-valid fourientations which do not contain any cycle made entirely of 2-way edges. 
		\item An \textit{$(A,B)$-valid cocycle class} is a cocycle equivalence class of $(A,B)$-valid fourientations which do not contain any $(A,B)$-cocycle made entirely of 0-way edges.
		\item An \textit{$(A,B)$-valid cycle-cocycle class} is a cycle-cocycle equivalence class of $(A,B)$-valid fourientations which does not contain any cycle made entirely of 2-way edges nor any $(A,B)$-cocycle made entirely of 0-way edges.
	\end{compactenum}
	For a set $S\subseteq E$ of edges, we denote by $\Four_{G,S}(A,B)/\cyc$ (resp. $\Four_{G,S}(A,B)/\coc$, $\Four_{G,S}(A,B)/\cc$) the set of $(A,B)$-valid cycle (resp. cocycle, cycle-cocycle) classes. 
\end{definition}
	
	Our main result is the following.
	
	\begin{theorem}\label{thm:eqClass}
	Let $G = (V,E)$ be a graph, and let $A,B\subseteq V^2$.
	\begin{compactenum}[(i)]
		\item $|\Four_{G,S}(A,B)/\cyc\,|$ is independent of $S$.
		\item $|\Four_{G,S}(A,B)/\coc\,|$ is independent of $S$.
		\item $|\Four_{G,S}(A,B)/\cc\,|$ is independent of $S$.
	\end{compactenum}
	\end{theorem}
	
	\newcommand{\un}[1]{\underline{#1}}
	Note that $\Four_{G,E}(A,B)/\cyc$ is the set of \emph{forests} of $G$ (subgraphs with no cycle). For a subgraph $F$ of $G$, let us denote by $F\cup \un{A}\cup\un{B}$ the graph obtained from $F$ by adding the set of edges $\{\{u,v\}\mid (u,v)\in A\cup B\}$. We say that $F$ is \emph{$(A,B)$-connected} if the connected components of $F\cup \un{A}\cup\un{B}$ and $G\cup \un{A}\cup\un{B}$ are equal. Then $\Four_{G,E}(A,B)/\coc$ is the set of $(A,B)$-connected subgraphs of $G$, while $\Four_{G,E}(A,B)/\cc$ is the set of $(A,B)$-connected forests of $G$. These sets of $(A,B)$-valid subgraphs are represented in Figure~\ref{fig:equivalence-classes}.

\fig{width=\linewidth}{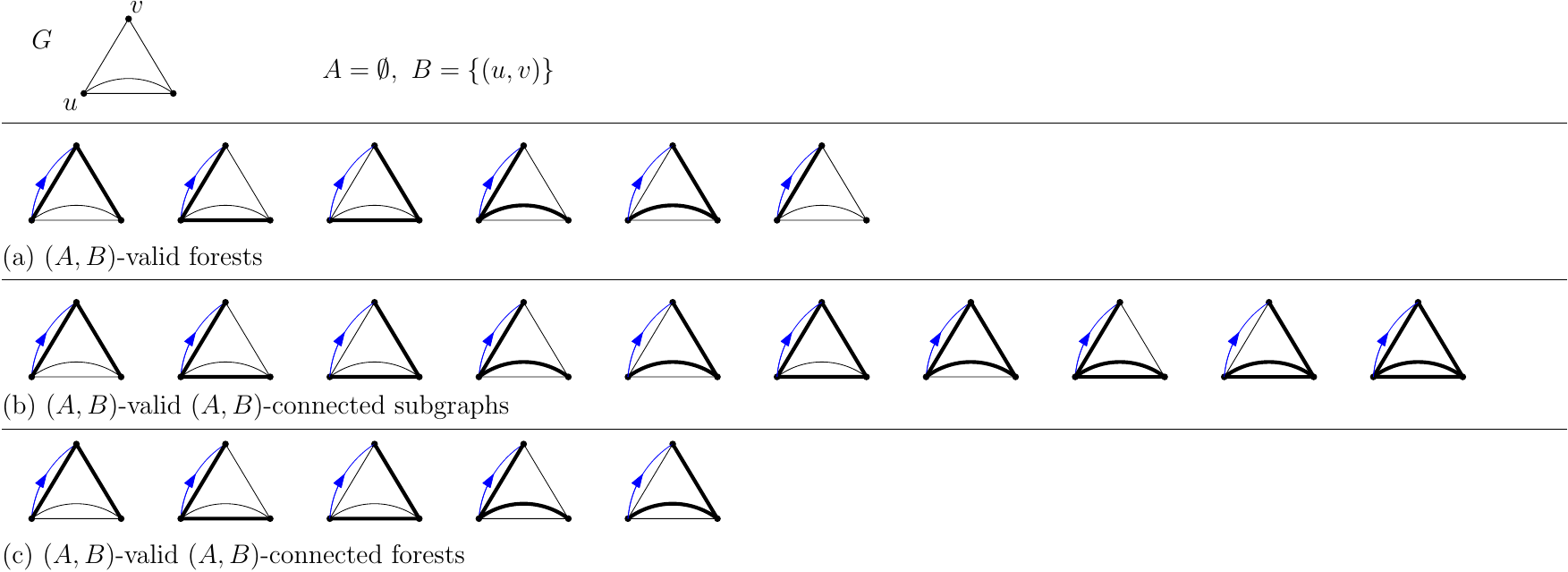}{(a) The $(A,B)$-valid forests (represented by thick edges). (b) The $(A,B)$-valid $(A,B)$-connected subgraphs. (c) The $(A,B)$-valid $(A,B)$-connected forests.}

Comparing the cases of $S=\emptyset$ and $S=E$ in Theorem~\ref{thm:eqClass} we obtain the following corollary:
\begin{corollary}\label{cor:equiv}
	Let $G = (V,E)$ be a graph, and let $A,B\subseteq V^2$. 
	\begin{compactenum}[(i)]
		\item The number of cycle equivalence classes of $(A,B)$-valid orientations is equal to the number of $(A,B)$-valid forests.
		\item The number of cocycle equivalence classes of $(A,B)$-valid orientations is equal to the number of $(A,B)$-valid $(A,B)$-connected subgraphs.
		\item The number of cycle-cocycle equivalence classes of $(A,B)$-valid orientations is equal to the number of $(A,B)$-valid $(A,B)$-connected forests.
	\end{compactenum}	
\end{corollary}

The results in Corollary~\ref{cor:equiv} for the case $A=B=\emptyset$ were first established by Gioan in~\cite{Gioan:degree-sequences}. Our proof techniques are entirely different. 
Let us also mention some other relevant references. The number of cycle/cocycle/cycle-cocycle equivalence classes of a graph $G$ are given by some evaluations of the Tutte polynomial.
It can be shown that the number of cycle equivalence classes of a graph $G$ is equal to the number of outdegree sequences (a.k.a. score vectors), and the fact that the number of outdegree sequences is equal to the number of forests was first established by Stanley in~\cite{Stanley:rationnal-polytopes} (see the survey \cite{EllisMonaghanMerino2010}).
With a bit of work, the number of cycle-cocycle equivalence classes of a graph $G$ can be related to the cardinality of its Jacobian (a.k.a. Picard group, sandpile group), and the fact that the cardinality of the Jacobian is equal to the number of spanning tree is the object of the celebrated Matrix-Tree theorem.

Before proceeding to the proof of Theorem \ref{thm:eqClass}, let us illustrate Corollary~\ref{cor:equiv} on a couple of examples.
	%

\begin{example}
One can check Corollary~\ref{cor:equiv} for the graph $G$ represented in Figures~\ref{fig:equivalence-classes} and~\ref{fig:AB-valid-subgraphs}. For instance, the cycle equivalence classes of $(A,B)$-valid orientations in Figure~\ref{fig:equivalence-classes} are equinumerous to the $(A,B)$-valid forests in Figure~\ref{fig:AB-valid-subgraphs}.
\end{example}
	
\begin{example}\label{ex:Aeq}
Fix a graph $G$ and two vertices $u,v$. Applying Corollary~\ref{cor:equiv}(i) with the sets $A =\lbrace (u,v)\rbrace $ and $B = \emptyset$, we get that the cycle equivalence classes of orientations of $G$ such that $v$ cannot reach $u$ is equal to the number of forests of $G$ such that $u$ and $v$ are in separate trees. This is illustrated in Figure~\ref{fig:cycle-A-single-arc}.
\end{example}
	
\fig{width = 0.8\textwidth}{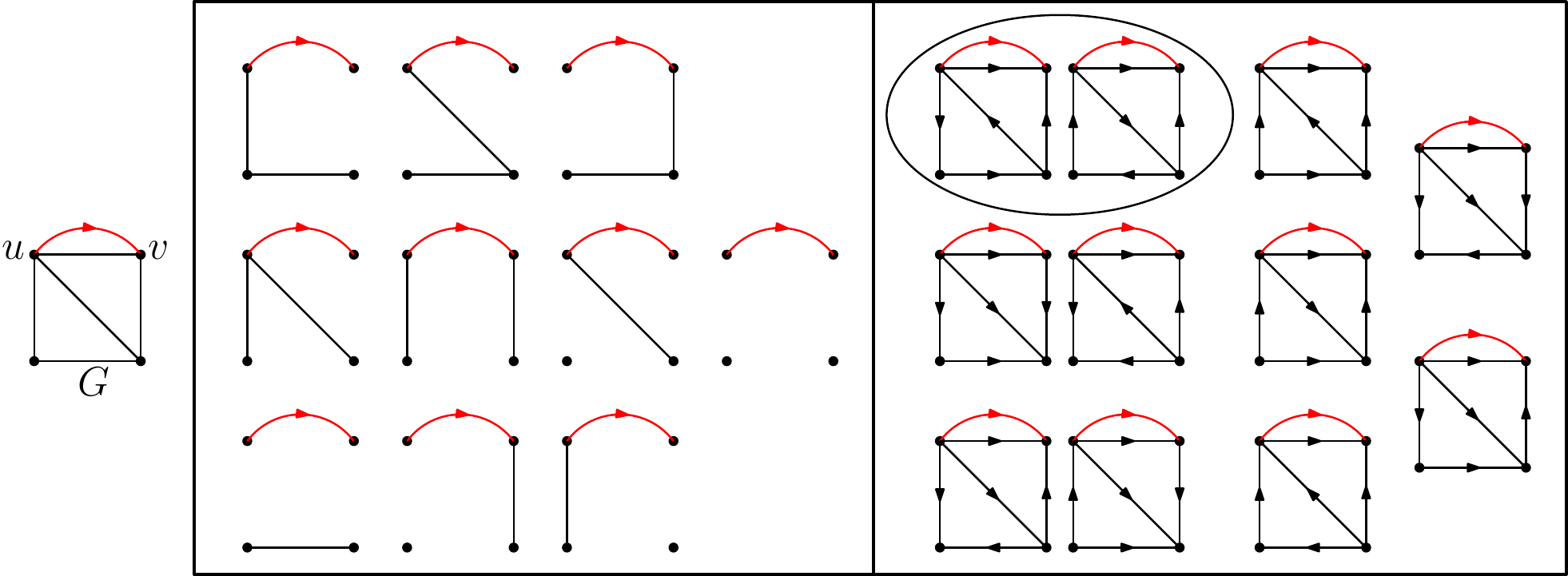}{An illustration of Corollary~\ref{cor:equiv}(i) for $A =\lbrace (u,v)\rbrace $ and $B = \emptyset$. The acyclic constraint $(u,v)\in A$ is represented by a red arc.
	Left: The forests for which $u$ and $v$ are in separate trees. Right: orientations in which $v$ cannot reach $u$. The two orientations which are cycle equivalent are circled.}

	%
	%
	%
	%
	%

\subsection{Proof of Theorem~\ref{thm:eqClass} for cycle-reversing classes}
Before starting the proof of Theorem~\ref{thm:eqClass}, we establish a result about cycle and cocycle equivalences which may be of independent interest.
\begin{lemma}\label{lem:reverse-once}
If two fourientations $\phi,\phi'$ are cycle (resp. cocycle, cycle-cocycle) equivalent, then $\phi$ can be obtained from $\phi'$ by repeatedly reversing some directed cycles (resp. directed $(A,B)$-cocycles, directed cycles or $(A,B)$-cocycles) which do not contain any 1-way edge having the same orientation in $\phi$ and $\phi'$.
\end{lemma}

The proof of Lemma \ref{lem:reverse-once} uses the following fact where we denote by $-a$ the arc opposite to~$a$.
\begin{fact} \label{fact:disjoint-decomposition} If two directed cycles (resp. cocycles) $C,C'$ have no common arc (but potentially some opposite arcs), then the set of arcs 
$\Delta=\{a\in C\mid -a\notin C'\}\cup \{a\in C'\mid -a\notin C\}$ can be obtained as a disjoint union of directed cycles (resp. cocycle).
\end{fact}
Although this fact sounds folklore (and is easy to establish, especially for cycles), the earlier references we could locate for the cocycle version are~\cite{Backman:partial, BBY:geometric-bijections,Ding:geometric-bijections}, the latter two applying to the larger context of regular oriented matroids. In this context, the property follows from interpreting $\Delta$ as a $\{0,1,-1\}^E$ vector in the cycle (resp. cocycle) space, and showing that such vectors are sums of disjoint directed cycles (resp. cocycles).

\begin{proof}[Proof of Lemma \ref{lem:reverse-once}]
Let us consider orientations first, and prove the stated property for two cycle equivalent orientations $\phi,\phi'$. 
By definition, there exists a sequence of directed cycles $C_1,\ldots,C_n$ such that one obtains $\phi'$ from $\phi$ by successively reversing $C_1,\ldots,C_k$. Consider such a sequence of directed cycle for which the total sum of the length of the cycles is minimal. We want to prove that the cycles $C_1,\ldots,C_k$ do not use any edge twice. Suppose for contradiction that there is $i<j$ such that the cycles $C_i$ and $C_j$ use the same edge, and take $j$ minimal with this property. By the minimality of $j$, the cycles $C_1,\ldots,C_{j-1}$ have no common edges, so their reversal can be done in any order. So we can now assume that $i=j-1$. Observe that $C_{j-1}$ and $C_j$ have no arc in common (since they are successive in the sequence), Fact~\ref{fact:disjoint-decomposition} ensures that the set 
$$C:=\{c\in C_{j-1}\mid -c\notin C_j\}\cup \{c\in C_{j}\mid -c\notin C_{j-1}\}$$
can be obtained as a disjoint union of directed cycles $C_1',\ldots,C_k'$. Observe now that one obtains $\phi'$ from $\phi$ by successively reversing the directed cycles $C_1,\ldots,C_{j-2}$ then $C_1',\ldots,C_k'$, and then $C_{j+1},\ldots,C_{n}$. This is in contradiction with our choice of the sequence $C_1,\ldots,C_n$ (minimal total length). This proves the property for cycle equivalent orientations. 

The proof for cocycle equivalent orientations is identical, up to replacing cycles by cocycles and again using Fact~\ref{fact:disjoint-decomposition}. 
(In our context we need to restrict our attention to $(A,B)$-cocycles of $G$, but these are the same as directed cocycles of the graph $G'=G\cup A\cup B$ which do not use the edges in $A,B$, so the fact and proof still hold.)
Lastly, if some orientations $\phi,\phi'$ are cycle-cocycle equivalent, then the cycle reversals and cocycle reversals act on disjoint set of edges:  cycle reversals act on cyclic edges while cocycle reversals act on acyclic edges (and the sets of cycle edges and acyclic edges are unchanged by reversals of directed cycles or directed cocycles), so the cycle and cocycle reversal are independent, and the above proof still applies.

Now consider the case of cycle (resp. cocycle, cycle-cocycle) equivalent fourientations $\phi,\phi'$. The set $R$ of 0-way edges and the set $T$ of 2-way edges are the same in $\phi$ and $\phi'$. Consider the orientations $\al=\phi\setminus R/T$ and $\al'=\psi\setminus R/T$ obtained by deleting the 0-way edges and contracting the 2-way edges. It is easy to see that the orientation $\al,\al'$ are cycle (resp. cocycle, cycle-cocycle) equivalent. By the above,  $\al$ can be obtained from $\al'$ by repeatedly reversing some disjoint directed cycles (resp. directed $(A,B)$-cocycles, directed cycles or $(A,B)$-cocycles) of $G\setminus R/T$. It is then clear that $\phi$ can be obtained from $\phi'$ by repeatedly reversing some directed cycles (resp. directed $(A,B)$-cocycles, directed cycles or $(A,B)$-cocycles) which do not contain any 1-way edge having the same orientation in $\phi$ and $\phi'$.
\end{proof}

We proceed to prove statement (i) of Theorem~\ref{thm:eqClass}. It suffices to show that for every set $S\subset E$ and edge $e\in E\setminus S$ one has 
\begin{equation}\label{eq:classes-e}
	|\Four_{G,S}(A,B)/\cyc| = |\Four_{G, S\cup\lbrace e\rbrace}(A,B)/\cyc|. 
\end{equation}
	
Fix $S$ and $e$ as above, and let $\mV:=\Four_{G,S}(A,B)/\cyc$ and $\mV':=\Four_{G, S\cup\lbrace e\rbrace}(A,B)/\cyc$. In order to prove that $|\mV|=|\mV'|$ we will now define a graph $\mathcal{G}_S(A,B)$ with vertex set~$\mV$. We say that two equivalence classes $\ov\phi$ and $\ov\psi$ in $\mV$ are \emph{$e$-adjacent} if there are fourientations $\phi'\in\ov\phi$ and $\psi'\in\ov\psi$ such that $\phi'$ and $\psi'$ coincide except that they have opposite orientations of $e$. Let $\mathcal{G}_S(A,B)$ be the graph (without multiple edges) with vertex set $\mV$, where two equivalence classes are connected by an edge if and only if they are $e$-adjacent.
The graph $\mathcal{G}_S(A,B)$ is represented in Figure~\ref{fig:G-class-combined-ex2}(a).


We make a few definitions and observations. For an equivalence class $\overline{\phi}\in \mV$ we denote by $\overline{\phi}_{\ra{e}}$ (resp. $\overline{\phi}_{\la{e}}$) the subset of fourientations in $\ov\phi$ containing $\ra{e}$ (resp. $\la{e}$). We call \emph{$\mV$-block} a non-empty subset of the form $\ov{\phi}_{\ra{e}}$ or $\ov{\phi}_{\la{e}}$ for some class $\ov\phi\in \mV$. In Figure~\ref{fig:G-class-combined-ex2}(a), the classes have been divided into their blocks.	We say that two $\mV$-blocks $B,B'$ are \emph{$e$-adjacent} if there exist fourientations $\phi\in B$ and $\psi\in B'$ such that $\phi$ and $\psi$ coincide except that they have opposite orientations of $e$. Note that two equivalence classes $\ov\phi$ and $\ov\psi$ in $\mV$ are $e$-adjacent if and only if a block of $\ov\phi$ is $e$-adjacent to a blocks of $\ov\psi$.
Lastly, observe that in each class $\overline{\phi}\in\mV$ the edge $e$ is either cyclic in every fourientation in $\ov \phi$, or acyclic in every fourientation in $\ov \phi$. We refer to the class $\ov \phi$ as either \textit{$e$-cyclic} or \textit{$e$-acyclic} accordingly.

\begin{lemma}\label{lem:degree}
Every $\mV$-block is $e$-adjacent to at most one other $\mV$-block. Consequently, every $e$-acyclic (resp. $e$-cyclic) equivalence class $\overline{\phi}$ in $\mV$ has degree at most 1 (resp. 2) in the graph $\mathcal{G}_S(A,B)$.
\end{lemma}

\begin{proof} Let $B$ be a $\mV$-block. Let $\phi$ and $\phi'$ be fourientations in $B$, and let $\psi$ and $\psi'$ be the fourientations obtained from $\phi$ and $\phi'$ by reversing the orientation of $e$. By Lemma~\ref{lem:reverse-once}, the fourientation $\phi$ can be obtained from $\phi'$ by reversing some directed cycles not containing $e$.
Hence, the same holds for $\psi$ and $\psi'$, which implies that they are cycle equivalent. Thus, if $\psi$ and $\psi'$ are $(A,B)$-valid then they belong to the same $\mV$-block (in fact, reversing $e$ gives a bijection between the $\mV$-block $B$ containing $\phi$ and the $\mV$-block $B'$ containing $\psi$). Therefore, $B$ is $e$-adjacent to at most one $\mV$-block.

The second statement follows by observing that $e$-acyclic classes have a single $\mV$-block, while the $e$-cyclic classes have two $\mV$-blocks. 
	\end{proof}
	
It follows from Lemma~\ref{lem:degree} that each connected component of the graph $\mathcal{G}_S(A,B)$ must be a path or a cycle. One can actually be a bit more precise: 	
	
	\begin{cor}\label{cor:path}
	Each connected component of $\mathcal{G}_S(A,B)$ is a either a path or a cycle of length 1 (that is, a single vertex $\ov\phi\in \mV$ incident to a loop).
	A class $\ov\phi\in \mV$ is in a cycle of length 1 in $\mathcal{G}_S(A,B)$ if and only if $e$ is the unique 1-way edge in a directed cycle of the fourientation $\phi$ (all the other edges on that cycle being 2-way edges). 
	\end{cor}
\begin{proof}
Consider a connected component $\gamma$ of the graph $\mathcal{G}_S(A,B)$, and let $\mU\subseteq \mV$ be the set of equivalence classes which are the vertices of $\gamma$. The set $T\subseteq S$ of 2-way edges is the same for every fourientation of every class in $\mU$. Let $(u,v)$ be the endpoints of the arc $\ra{e}$ in the contracted graph $G/T$.

Suppose first that $u=v$, which occurs when $\ra{e}$ is the only 1-way edge in a directed cycle of the fourientation in $\mU$. In this case, it is easy to see that $\mU$ is just one class with a loop in $\mathcal{G}_S(A,B)$. 

Suppose now that $u\neq v$. We want to show that the connected component $\gamma$ is not a cycle. For a fourientation $\phi$ in one of the classes in $\mU$, we define the \emph{$u$-outdegree} of $\phi$ to be the outdegree of the vertex $u$ (the number of 1-way edge with origin $u$) in the contracted fourientations $\phi/T$.
	For a class $\ov\phi\in \mU$, the $u$-outdegree is the same for every fourientation $\phi'$ in $\phi$, since it is unchanged by a cycle reversal. However, the $u$-outdegree will change by 1 when going from a class $\ov\phi \in\mU$ to an $e$-adjacent class (it increases by 1 when going from a class $\ov \phi$ to an $e$-adjacent class $\ov \psi$ obtained by replacing $\la{e}$ by $\ra{e}$). In fact the $u$-outdegree changes monotonously along a simple path of $\gamma$. Thus, $\gamma$ cannot contain a cycle. 
\end{proof}
	%
	%
	%

We will now prove that $|\mV|=|\mV'|$ by describing how the cycle equivalence classes change (merge and split) from $\mV$ to $\mV'$ when the 1-way arc $\ra{e}$ or $\la{e}$ is replaced with a solid edge $\ba{e}$ or $\ua{e}$. This is illustrated in Figure~\ref{fig:G-class-combined-ex2}. 
	
\fig{width = \textwidth}{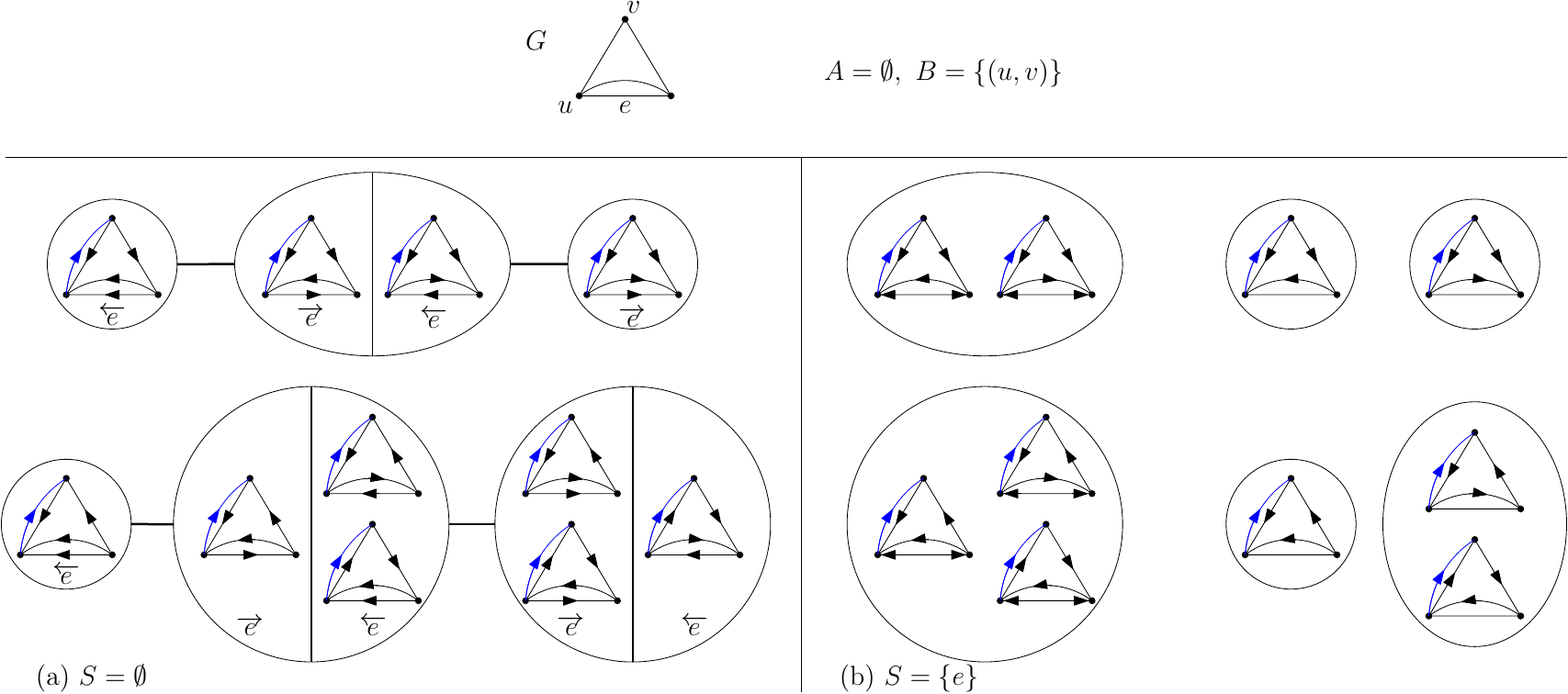}{Correspondence between the equivalences classes in $\mV$ and in $\mV'$. Here $S=\emptyset$, $A=\emptyset$, $B=\{(u,v)\}$. The cyclic constraint $(u,v)\in B$ is indicated by a blue arc, and the 0-way edges are indicated as line segments with no arrow.
(a) The graph $\mathcal{G}_S(A,B)$ whose vertex set $\mV$ correspond to the six cycle equivalence classes of $(A,B)$-valid orientations. The partition of $e$-cyclic equivalence classes into their blocks is indicated.
(b) The resulting six equivalence classes in $\mV'$. In this example, the two $\ba{e}$-classes corresponds to the two connected components of $\mathcal{G}_S(A,B)$, while the four $\ua{e}$-classes  correspond to the four edges of $\mathcal{G}_S(A,B)$.}

Let us start with establishing some vocabulary. Let us call \emph{$\ua{e}$-class} an equivalence class in $\mV'$ with where $e$ has the configuration $\ua{e}$. 
Given a $\mV$-block $B$, all the fourientations obtained from those in $B$ by giving $e$ the configuration $\ua{e}$ are all cycle equivalent. Hence, if they are $(A,B)$-valid, they belong to the same $\ua{e}$-class $C$, and in this case we say that the $\mV$-block $B$ is \emph{associated} to $C$. We define similarly the $\ba{e}$-classes and the $\mV$-blocks associated to an $\ba{e}$-class. Trivially, if some $\mV$-blocks $B,B'$ are $e$-adjacent, then $B$ is associated to a class $C\in \mV'$ if and only if $B'$ is associated to the class $C$.

%
%
%
%

By \eqref{eq:suffice2}, for any class in $\mV'$ there is at least one $\mV$-block associated to it. We give a more precise result.

\begin{lemma}\label{lem:block-association}
Each $\ua{e}$-class $C\in \mV'$ is associated either to a single $\mV$-block or to a pair of $e$-adjacent $\mV$-blocks.
Each $\ba{e}$-class $C\in \mV'$ is associated to a set of $\mV$-blocks $\{B_1,\ldots,B_n\}$, where for all $i\in [n-1]$ the $\mV$-blocks $B_{i},B_{i+1}$ are either in the same $e$-cyclic class or are $e$-adjacent.
\end{lemma}

\begin{proof}
Consider first a $\ua{e}$-class $C\in\mV'$. The fourientations obtained from those in $\ov\phi$ by replacing $\ua{e}$ by $\ra{e}$ (resp. $\la{e}$) are all cycle equivalent, thus they belong to a single $\mV$-block. Hence, $C$ is either associated to a single $\mV$-block, or a pair of $e$-adjacent $\mV$-blocks.

Consider now a $\ba{e}$-class $C\in\mV'$. Let $B,B'$ be $\mV$-blocks associated to $C$. Let $\phi\in B$ and $\phi'\in B'$. By definition the fourientations $\psi,\psi'$ obtained from $\phi,\phi'$ by giving $e$ the configuration $\ba{e}$ are cycle equivalent. Hence, there exists a sequence $\psi_1=\psi,\psi_2,\ldots,\psi_k=\psi'$ of fourientations such that for all $i$, $\psi_{i+1}$ is obtained from $\psi_i$ by reversing a directed cycle. This easily implies that there exists a sequence $\phi_1=\phi,\phi_2,\ldots,\phi_\ell=\phi'$ of fourientations where for all $i$, $\phi_i$ and $\phi_{i+1}$ are in the same $\mV$-block, or in different $\mV$-blocks of the same $e$-cyclic class, or in $e$-adjacent $\mV$-blocks. This proves the stated result.
\end{proof}

By Lemma~\ref{lem:block-association}, the set of blocks associated to a given class in $\mV'$ are all part of the same connected component of the graph $\mG_S(A,B)$. We now consider a connected component $\ga$ of $\mathcal{G}_S(A,B)$, and proceed to count the classes in $\mV'$ associated with the blocks in $\gamma$. 

Consider first the case where the connected component $\gamma$ is an isolated vertex $\ov\phi\in \mV$ incident to a loop. By Corollary~\ref{cor:path}, the edge $e$ is the only 1-way edge in a directed cycle of the fourientations in $\ov\phi$. Since cycles of 2-way edges are disallowed for $(A,B)$-valid cycle classes (see Definition~\ref{def:valid-classes}(i)), only the configuration $\ua{e}$ is possible to get a class in $\mV'$ (although the configurations $\ua{e}$ and $\ba{e}$ both give $(A,B)$-valid fourientations by~\eqref{eq:suffice2}). Hence, in this case the blocks in $\gamma$ are associated to a single class in $\mV'$.

We now consider a connected component $\gamma$ which is a path. Let $B$ be a $\mV$-block in $\gamma$, and let $\phi$ be a fourientation in $B$. By Corollary~\ref{cor:path}, the edge $e$ is not the only 1-way edge in a directed cycle of $\phi$. Hence, by~\eqref{eq:suffice2}, the block $B$ is associated to two classes in $\mV'$ if $B$ is $e$-adjacent to another $\mV$-block, and $B$ is associated to a single class in $\mV'$ otherwise.


Suppose that the path $\gamma$ has $k\geq 2$ vertices. As explained above, the $\mV$-blocks in $\gamma$ incident to an edge of $\mathcal{G}_S(A,B)$ are associated to both a $\ua{e}$-class and a $\ba{e}$-class. In fact, it follows from Lemma~\ref{lem:block-association} that all the blocks in $\gamma$ are associated to the same $\ba{e}$-class, while each pair of $e$-adjacent blocks is associated to a distinct $\ua{e}$-class. Thus, the blocks in $\gamma$ are associated to 1 $\ba{e}$-class and $(k-1)$ $\ua{e}$-classes. This is illustrated in Figure~\ref{fig:G-class-combined-ex2}.

Lastly, if the path $\gamma$ has a single vertex, then the $\mV$-blocks in $\gamma$ are associated to a single class in $\mV'$ (precisely, a $\ba{e}$-class if the class in $\gamma$ is $e$-cyclic and a $\ua{e}$-class otherwise, by Lemma~\ref{lem:upgrade}). 

In every case, the number of classes in $\mV'$ associated to the blocks in $\gamma$ is equal to the number of vertices of $\gamma$. This proves $|\mV|=|\mV'|$, and completes the proof of Theorem~\ref{thm:eqClass}(i).

\subsection{Proof of of Theorem~\ref{thm:eqClass} for cocycle equivalence classes}	
We now focus on proving Theorem~\ref{thm:eqClass}(ii), following a similar strategy as for Theorem~\ref{thm:eqClass}(i). We fix a subset $S\subset E$ and an edge $e\in E\setminus S$ and we let $\mV := \Four_{G,S}(A,B)/\coc$ and $\mV' := \Four_{G,S\cup \lbrace e\rbrace}(A,B)/\coc$. We need to show that $|\mV|=|\mV'|$. 

We use the same terminology as before regarding the classes in $\mV$ and their block ($e$-cyclic, $e$-acyclic, $e$-adjacent). We let $\mathcal{G}'_S(A,B)$ be the graph with vertex set $\mV$, where two classes in $\mV$ are connected by an edge if and only if they are $e$-adjacent. We make a couple of observations, which are analogous to the case of cycle equivalence.
	\begin{lemma}\label{lem:degree-ii}
	Every $\mV$-block is $e$-adjacent to at most one other $\mV$-block. Consequently, every $e$-cyclic (resp. $e$-acyclic) equivalence class $\overline{\phi}$ in $\mV$ has degree at most 1 (resp. 2) in~$\mathcal{G}_S'(A,B)$.
	\end{lemma}

Lemma~\ref{lem:degree-ii} can be proved analogously to Lemma~\ref{lem:degree} since Lemma~\ref{lem:reverse-once} still implies that each $\mV$-block is $e$-adjacent to at most one other $\mV$-block. 
It follows from Lemma~\ref{lem:degree-ii} that each connected component of $\mathcal{G}_S'(A,B)$ is either a path or a cycle. One can actually be a bit more precise: 	

	\begin{cor}\label{cor:path-ii}
		Each connected component of $\mathcal{G}'_S(A,B)$ is either a path or a cycle of length 1 (that is, a vertex $\overline{\phi}\in \mV$ with a loop).
		A class $\ov\phi\in \mV$ is in a cycle of length 1 in $\mathcal{G}'_S(A,B)$ if and only if $e$ is the unique 1-way edge in a directed $(A,B)$-cocycle of the fourientation $\phi$ (all the other edges on that cocycle being 0-way edges).
	\end{cor}
	
\begin{proof}
Consider a connected component $\gamma$ of the graph $\mathcal{G}'_S(A,B)$. Clearly the set $R\subseteq S$ of 0-way edges is the same for every fourientation $\phi$ in every class $\ov \phi$ in~$\gamma$. We define $G'=(G\setminus R)\cup \un{A}\cup\un{B}$ as the graph obtained from $G$ by deleting the edges in $R$ and adding the edges in $\{\{u,v\}\mid (u,v)\in A\cup B\}$.

Suppose first that the edge $e$ is an isthmus in $G'$. In this case $e$ belongs to a directed $(A,B)$-cocycle of $G$ containing no other 1-way edge. Hence, the connected component $\ga$ is made of a single equivalence $\ov\phi$ with a loop incident to it.

Suppose now that the edge $e$ is not an isthmus in $G'$. We want to show that the connected component $\gamma$ of $\mathcal{G}'_S(A,B)$ is not a cycle. Consider a cycle $C$ of $G'$ containing $e$. We designate a canonical direction for the cycle $C$, and for any fourientation $\phi$ of $G$, we call the \emph{$C$-flow} of $\phi$ the number of 1-way edges of $\phi$ on $C$ whose orientation coincides with the canonical orientation of $C$. We observe that, for any fourientation $\phi$ in a class in $\gamma$, the $C$-flow is unchanged by the reversal of any directed cocycle $D$ of $\phi$. This is because $D$ must intersect the cycle $C$ an even number of times (always at 1-way edges), and the number of 1-way edges in $D$ oriented in agreement with the canonical orientation of $C$ is equal to the number of 1-way edges in $D$ oriented in disagreement with the canonical direction of $C$. Thus, the $C$-flow is the same for every fourientation in the cocycle class $\overline{\phi}$. However, the $C$-flow will change by 1 when going from a class $\overline{\phi}$ in $\gamma$ to an $e$-adjacent class. In fact, the $C$-flow will change monotonously along a simple path of $\mathcal{G}'_S(A,B)$. This shows that the connected component $\gamma$ is not a cycle, as wanted.
	\end{proof}

The proof of Theorem~\ref{thm:eqClass}(ii) is analogous to the proof of Theorem~\ref{thm:eqClass}(i). Here are the main differences. Corollary~\ref{cor:path} replaces Corollary~\ref{cor:path-ii}, and Lemma~\ref{lem:block-association-ii} is replaced by the following:
\begin{lemma}\label{lem:block-association-ii}
Each $\ba{e}$-class $C\in \mV'$ is associated either to a single $\mV$-block or a pair of $e$-adjacent $\mV$-blocks.
Each $\ua{e}$-class $C\in \mV'$ is associated to a set of $\mV$-blocks $\{B_1,\ldots,B_n\}$, where for all $i\in [n-1]$ the blocks $B_{i},B_{i+1}$ are either in the same $e$-acyclic class or are $e$-adjacent.
\end{lemma}
The proof of Lemma~\ref{lem:block-association-ii} is the same (mutatis mutandis) as that of Lemma~\ref{lem:block-association}. By combining Equation \eqref{eq:suffice2}, Corollary~\ref{cor:path-ii} and Lemma~\ref{lem:block-association-ii}, one can then show that each connected component $\gamma$ of $\mG_S'(A,B)$ with $k$ classes of $\mV$ gives raise to $k$ classes in $\mV'$. This can be done by considering separately three different cases.
\begin{compactitem}
\item If $\gamma$ is an isolated vertex $\ov\phi\in \mV$ incident to a loop, then the edge $e$ is the only 1-way edge in a directed $(A,B)$-cocycle of the fourientations $\psi\in \ov\phi$. Since an entire $(A,B)$-cocycle of 0-way edges is disallowed for $(A,B)$-valid classes, we may only replace $e$ with the configuration $\ba{e}$, thus the blocks in $\gamma$ are associated with a single class in $\mV'$.
\item If $\gamma$ is a path with $k\geq 2$ vertices, then the blocks in $\gamma$ are associated to 1 $\ua{e}$-class and $k-1$ $\ba{e}$-class (one for each edge of $\gamma$).
\item If $\gamma$ is a path with a single vertex, then the blocks in $\gamma$ are associated to a single class of $\mV'$.
\end{compactitem}
In every case, the number of classes in $\mV'$ associated to the blocks in $\gamma$ is equal to the number of vertices of $\gamma$.
This proves $|\mV|=|\mV'|$, hence Theorem~\ref{thm:eqClass}(ii).

\subsection{Proof of Theorem~\ref{thm:eqClass} for cycle-cocycle equivalence classes}	
	We now prove Theorem~\ref{thm:eqClass}(iii). We fix a subset $S\subset E$ and an edge $e\in E\setminus S$ and we let $\mV := \Four_{G,S}(A,B)/\cc$, and $\mV' := \Four_{G,S\cup\lbrace e\rbrace}(A,B)/\cc$. We define $\mV$-blocks and $e$-adjacency as before. Let $\mG_S''(A,B)$ be the graph with vertex set $\mV$ with edges corresponding to $e$-adjacency. We observe that each class in $\mV$ contains two $\mV$-blocks. 

	\begin{lemma}\label{lem:degree-iii}
	Every $\mV$-block is $e$-adjacent to at most one other $\mV$-block. Consequently, every equivalence class $\overline{\phi}$ in $\mV$ has degree at most 2 in the graph $\mathcal{G}_S''(A,B)$.
	\end{lemma}

Lemma~\ref{lem:degree-iii} can be proved analogously to Lemma~\ref{lem:degree} based on Lemma~\ref{lem:reverse-once}. It implies that each connected component of the graph $\mathcal{G}_S''(A,B)$ must be a path or a cycle. We will now prove the following more precise statement.	

\begin{cor}\label{cor:path-iii}
		Each connected component of $\mathcal{G}''_S(A,B)$ is either a path or a cycle. 
Each cycle of $\mathcal{G}''_S(A,B)$ of length greater than 1 contain both $e$-cyclic classes and $e$-acyclic classes.
A class $\ov\phi\in \mV$ is in a cycle of length 1 in $\mathcal{G}_S''(A,B)$ if and only if $e$ is the unique 1-way edge in a directed cycle of $\phi$ (all the other edges on that cycle being 2-way edges) or the unique 1-way edge in a directed $(A,B)$-cocycle of $\phi$ (all the other edges on that cocycle being 0-way edges) . 
\end{cor}
\begin{proof}
Consider a connected component $\gamma$ of the graph $\mathcal{G}_S(A,B)$, and let $\mU\subseteq \mV$ be the set of equivalence classes which are the vertices of $\gamma$. Every fourientation of every class in $\mU$ has the same set $R\subseteq S$ of 0-way edges and the same set $T\subseteq S$ of 2-way edges.

If $\{e\}\cup T$ contains a cycle of $G$, then it is easy to see that the component $\gamma$ is made of a unique ($e$-cyclic) class incident to a loop. 
If $\{e\}\cup R$ contains a $(A,B)$-cocycle of $G$, then the component $\gamma$ is made of a unique ($e$-acyclic) class incident to a loop. 

We now suppose that $\{e\}\cup T$ does not contain a cycle of $G$ and $\{e\}\cup R$ does not contain a $(A,B)$-cocycle of $G$. By reasoning as in the proof of Lemma~\ref{lem:degree} (and considering the $u$-outdegree, where $u$ is an endpoint of $e$ in $G/T$) one can show that $\gamma$ is not a cycle of $e$-cyclic classes (because the $u$-outdegree is constant over $e$-cyclic classes, but changes monotonously when walking along a path of $e$-cyclic classes in $\gamma$). By reasoning as in the proof of Lemma~\ref{lem:degree-ii} (and considering the $C$-flow, where $C$ is a cycle $G\setminus R$) one can show that $\gamma$ is not a cycle of $e$-acyclic classes (because the $C$-flow is constant over $e$-acyclic classes, but changes monotonously when walking along a path of $e$-acyclic classes in $\gamma$). This completes the proof.
\end{proof}

We wish to prove $|\mV| = |\mV'|$. Again, this will be done by describing how the equivalence classes change when the 1-way arc $\ra{e}$ or $\la{e}$ is replaced by $\ba{e}$ or $\ua{e}$. By reasoning similarly as before we get the following result.
\begin{lemma}\label{lem:block-association-iii}
Each $\ba{e}$-class (resp. $\ua{e}$-class) in $\mV'$ is associated to a set of $\mV$-blocks $\{B_1,\ldots,B_n\}$, where for all $i\in [n-1]$ the $\mV$-blocks $B_{i},B_{i+1}$ are either in the same $e$-cyclic (resp. $e$-acyclic) class or are $e$-adjacent.
\end{lemma}

\fig{width = \textwidth}{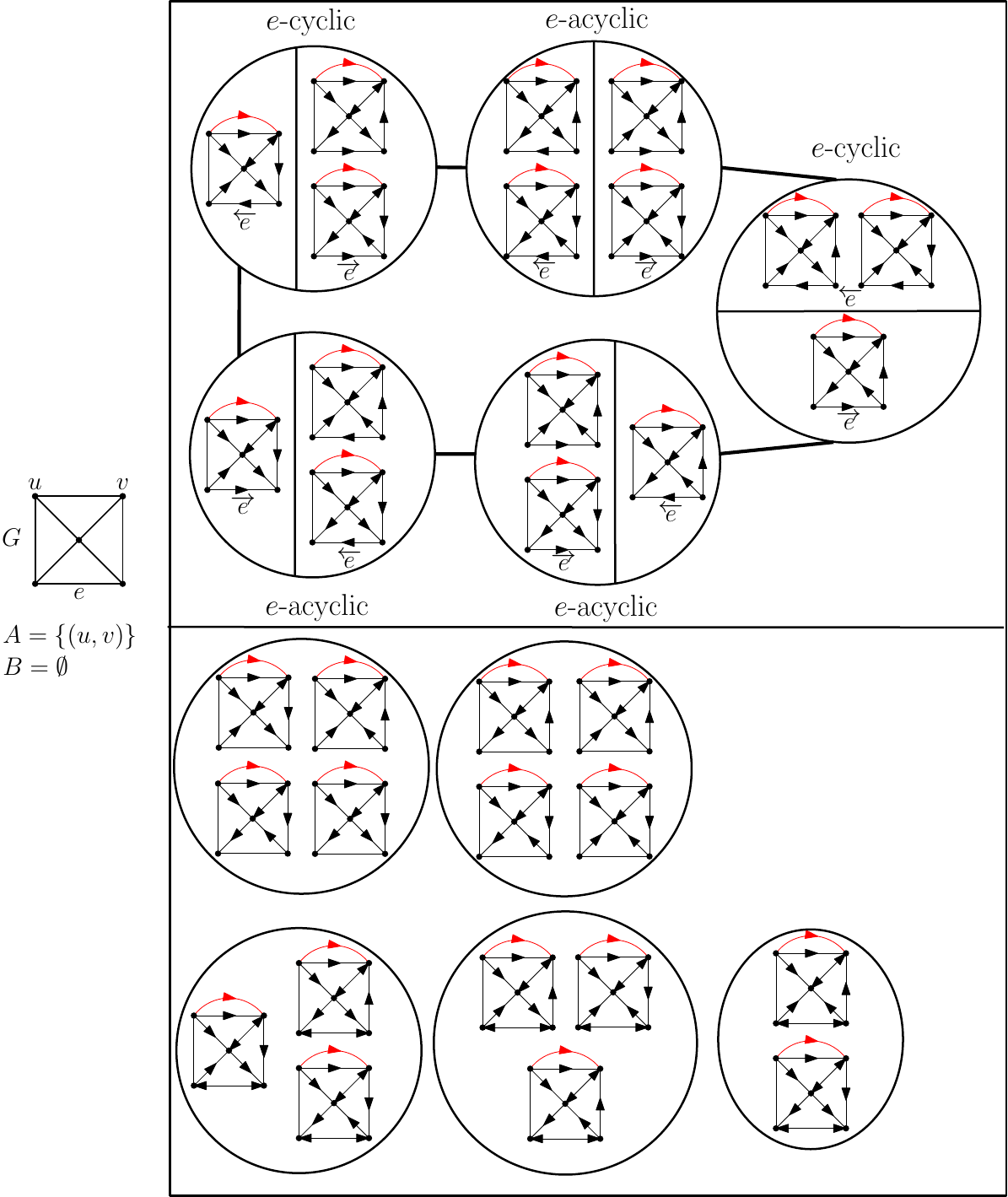}{Correspondence between cycle-cocycle equivalence classes in $\mV$ and $\mV'$. The acyclic constraint $(u,v)\in A$ is drawn in red. Top: A connected component $\gamma$ of the graph $\mathcal{G}_S''(A,B)$ of cycle-cocycle equivalence classes. This connected component of $\ga$ is a cycle with $a=2$ $e$-cyclic classes and $b=3$ $e$-acyclic classes. Bottom: The resulting cycle-cocycle classes in $\mV'$. As claimed, there are $a=2$ resulting  $\ua{e}$-classes, $b=3$ resulting $\ba{e}$-classes. }

By combining Equation \eqref{eq:suffice2}, Corollary~\ref{cor:path-iii} and Lemma~\ref{lem:block-association-iii}, one can show that each connected component $\gamma$ of $\mG_S'(A,B)$ with $k$ classes of $\mV$ gives raise to $k$ classes in $\mV'$. This is done by considering separately several cases.
 
The case where $\gamma$ is a cycle of length 1 (a vertex incident to a loop) must be treated separately. Indeed, by Corollary~\ref{cor:path-iii} this is the case where the edge $e$ is the only 1-way edge in either a directed cycle or a directed $(A,B)$-cocycle of the fourientations in $\ga$. In this case there is exactly one choice of $\ba{e}$ or $\ua{e}$ which does not result in a cycle of 2-way edges or a $(A,B)$-cocycle of 0-way edges, hence the blocks in $\gamma$ are associated to a single class in $\mV'$.
 
In all the other cases, by \eqref{eq:suffice2}, the blocks in $\ga$ are associated to two classes in $\mV'$ if they are $e$-adjacent to another $\mV$-block in $\ga$, and associated to a single class in $\mV'$ otherwise. We consider several cases.
\begin{compactitem}
\item Suppose that $\gamma$ is a cycle with $k\geq 2$ vertices. See Figure~\ref{fig:cycle-cocycle-merge2} for an illustration of this case. 
By Corollary~\ref{cor:path-iii}, there are both $e$-cyclic and $e$-acyclic classes in $\ga$.
Moreover, as explained above, each $\mV$-block in $\gamma$ is associated to both a $\ba{e}$-class and a $\ua{e}$-class in $\mV'$.
It is easy to see from Lemma~\ref{lem:block-association-iii} that the $\ba{e}$-classes (resp. $\ua{e}$-classes) associated to the blocks in $\gamma$ are in one-to-one correspondence with the ``subpaths of $\gamma$'' obtained from the cycle $\gamma$ by ``cutting in two'' the $e$-acyclic (resp. $e$-cyclic) classes (that is, separating the two blocks of these classes). 
Thus, if $\gamma$ has $a$ $e$-acyclic classes and $b$ $e$-cyclic classes, then the blocks in $\gamma$ are associated with $a$ $\ba{e}$-classes and $b$ $\ua{e}$-classes in $\mV'$. This gives a total of $a+b=k$ classes in $\mV'$ associated with the $\mV$-blocks in $\ga$. 
\item Suppose now that $\gamma$ is a path with $k\geq 2$ vertices. Each $\mV$-block in $\ga$ which is adjacent to an edge of $\gamma$ is associated to both a $\ba{e}$-class and a $\ua{e}$-class in $\mV'$. 
It is easy to see from Lemma~\ref{lem:block-association-iii} that the $\ba{e}$-classes (resp. $\ua{e}$-classes) associated to the blocks in $\gamma$ are in one-to-one correspondence with the ``subpaths of $\gamma$'' obtained from the path $\gamma$ by cutting in two the $e$-acyclic (resp. $e$-cyclic) classes of degree 2 in $\gamma$ (separating their two blocks). 
Thus, if $\gamma$ has $a$ $e$-acyclic classes of degree 2 and $b$ $e$-cyclic classes of degree 2, then the blocks in $\gamma$ are associated with $a+1$ $\ba{e}$-classes and $b+1$ $\ua{e}$-classes in $\mV'$. This gives a total of $a+b+2=k$ classes in $\mV'$ associated with the $\mV$-blocks in $\ga$.
\item Lastly, if the path $\gamma$ has a single vertex, then the $\mV$-blocks in $\gamma$ are associated to a single class in $\mV'$ (precisely, a $\ba{e}$-class if the class in $\gamma$ is $e$-cyclic and a $\ua{e}$-class otherwise, by Lemma~\ref{lem:upgrade}). 
\end{compactitem}
In every case, the number of classes in $\mV'$ associated to the blocks in $\gamma$ is equal to the number of vertices of $\gamma$.
This proves $|\mV|=|\mV'|$, hence Theorem~\ref{thm:eqClass}(iii).

\bigskip

\noindent \textbf{Acknowledgments:} We are extremely grateful to Ira Gessel who, by asking us whether it was possible to interpret~\eqref{eq:Ira} combinatorially, led us on the path to discovering the results in the present article. We also thank Neha Goregaokar, Emeric Gioan, Alex Leighton and Vasiliy Nekrasov for stimulating discussions.

\bibliographystyle{abbrv}
\bibliography{biblio-orient-subgraphs}
	
\end{document}